\documentclass[11pt,reqno]{amsart}

\usepackage[T1]{fontenc}
\usepackage{mathptmx,microtype,hyperref,graphicx,dsfont}
\usepackage{amsthm,amsmath,amssymb}
\usepackage[foot]{amsaddr}
\usepackage{enumitem}
\usepackage[nameinlink,capitalise]{cleveref}
\usepackage[font=footnotesize,margin=2cm]{caption}

\crefname{theorem}{Theorem}{Theorems}
\crefname{thm}{Theorem}{Theorems}
\crefname{mainthm}{Theorem}{Theorems}
\crefname{lemma}{Lemma}{Lemmas}
\crefname{lem}{Lemma}{Lemmas}
\crefname{remark}{Remark}{Remarks}
\crefname{claim}{Claim}{Claims}
\crefname{subclaim}{Sub-claim}{Sub-claims}
\crefname{prop}{Proposition}{Propositions}
\crefname{proposition}{Proposition}{Propositions}
\crefname{defn}{Definition}{Definitions}
\crefname{corollary}{Corollary}{Corollaries}
\crefname{conjecture}{Conjecture}{Conjectures}
\crefname{question}{Question}{Questions}
\crefname{chapter}{Chapter}{Chapters}
\crefname{section}{Section}{Sections}
\crefname{figure}{Figure}{Figures}
\newcommand{\crefrangeconjunction}{--}

\theoremstyle{plain}

\newtheorem{thm}{Theorem}
\newtheorem*{thm*}{Theorem}

\newtheorem{lem}[thm]{Lemma}

\newtheorem{corollary}[thm]{Corollary}

\theoremstyle{definition}
\newtheorem{defn}[thm]{Definition}

\theoremstyle{remark}

\renewcommand{\P}{{\bf P}}
\newcommand{\E}{{\bf E}}

\newcommand{\Z}{{\mathbb Z}}

\newcommand{\G}{{\mathcal G}}
\newcommand{\Gnp}{\G_{n,p}}

\newcommand{\1}{{\bf 1}}

\author[O. Angel]{Omer Angel}
\address{Department of Mathematics, University of British Columbia}
\email{angel@math.ubc.ca}

\author[B. Kolesnik]{Brett Kolesnik}
\address{Department of Statistics, University of California, Berkeley}
\email{bkolesnik@berkeley.edu}

\begin{document}

\title[Large deviations for subcritical bootstrap percolation]
{Large deviations for subcritical bootstrap percolation on the random graph}

\begin{abstract}
We study atypical behavior in bootstrap percolation 
on the Erd\H{o}s--R\'enyi random graph. 
Initially a set $S$ is infected. 
Other vertices are infected once at least $r$ of their
neighbors become infected.  
 Janson et al.\ (2012)
locates the critical size of $S$, above which
it is likely that the infection will spread almost everywhere. 
Below this threshold, a central limit theorem is proved
for the size of the eventually infected set.
In this note, we calculate  
the rate function for the event 
that a small set $S$ eventually infects an unexpected number of vertices,
and identify the least-cost trajectory 
realizing such a large deviation. 
\end{abstract}

\maketitle

\section{Introduction}
\label{S_intro}

Bootstrap percolation was originally proposed by physicists \cite{PRK75,CLR79}
to model
the phase transition observed in disordered magnets. 
Since then a large literature has developed, motivated 
by beautiful results, e.g.\  \cite{AL88,S92,H03,BBDCM12}, 
and a variety of applications across many fields, see e.g.\ \cite{J91,JL03} and references therein. 

In this work, we consider the spread of an infection by the $r$-neighbor
bootstrap percolation dynamics on the Erd{\H{o}}s--R{\'e}nyi~\cite{ER59}
graph $\Gnp$, in which any two vertices in $[n]$ are neighbors independently
with probability $p$. However, we think our methods are much more 
widely applicable. 
Initially some subset $S_0\subset[n]$ is infected. 
Other vertices are infected once at least $r$ of their
neighbors become infected. 

Most of the literature has focused on the typical behavior; 
specifically, the critical size at which point a uniformly
random initial set $S_0$ is likely to infect most of the graph.  
Less is known about the atypical behavior, such
as when a small set $S_0$ is capable of eventually infecting 
many more vertices than expected (e.g.\ influencers 
or superspreaders in a social network). 
 
For analytical convenience, we rephrase the dynamics in terms of an exploration process 
(cf.\ \cite{S83,ST85,JLTV12})
in which vertices are infected one at a time. At any given step, 
vertices are either 
{\it susceptible,} {\it infected} or {\it healthy}. All susceptible vertices become infected eventually,
and then remain infected. 
When a vertex is infected, other currently healthy vertices may become susceptible.   
The process ends once a stable configuration has been reached in which 
no vertices are susceptible.  

More formally, at each step $t$, there are sets $I_t$ and $S_t$ of
infected and susceptible vertices. 
Initially, $I_0=\emptyset$.
In step $t\ge1$, some vertex $v_t\in S_{t-1}$
is infected. All remaining edges from $v_t$ are revealed. 
To obtain $S_t$ from $S_{t-1}$, we remove $v_t$ and add
all neighbors of $v_t$ with exactly $r-1$ neighbors in $I_{t-1}$. 
We then add $v_t$ to $I_{t-1}$ to obtain $I_t$.
The process ends at step $t_*=\min\{t\ge1:S_t=\emptyset\}$ when no further
vertices can be infected. Let $I_*=I_{t_*}$
denote the eventually infected set. Clearly, $I_*$
does not depend on the order in which vertices are infected. 

Janson et al.\ \cite{JLTV12} 
(cf.\ Vallier's thesis~\cite{V07})
 identifies the critical size of $S_0$, 
for all $r\ge2$
and 
\begin{equation}\label{E_p}
p=((r-1)!/n)^{1/r}\vartheta^{1/r-1}, \quad 1\ll \vartheta\ll n. 
\end{equation}
The extreme cases 
$p\sim c/n$ and $p\sim c/n^{1/r}$
are also addressed in \cite{JLTV12}, where the model
behaves differently. 
We assume \eqref{E_p} throughout this work.

A sharp threshold is observed in \cite{JLTV12}. If
more than $(1-1/r)\vartheta$ vertices are initially susceptible, then 
all except $o(n)$ many vertices are eventually infected. 
Otherwise, 
the eventually infected set is much smaller, of size $O(\vartheta)\ll n$. 

\begin{thm}[\cite{JLTV12} Theorem 3.1]
\label{T_JLTV}
Let $p$ be as in \eqref{E_p} and $\alpha\ge0$. 
Put $\alpha_r=(1-1/r)\alpha$.
Suppose that a set $S_0$ (independent of $\Gnp$) 
of size $|S_0|\sim\alpha_r\vartheta$
is initially susceptible.
If $\alpha>1$, then with high probability 
$|I_*|\sim n$.
If $\alpha<1$, then with high probability 
$|I_*|\sim\varphi_\alpha \vartheta$, where
$\varphi_\alpha\in[\alpha_r,\alpha]$ uniquely satisfies
\begin{equation}\label{E_varphi}
\varphi_\alpha-\varphi_\alpha^r/r=\alpha_r. 
\end{equation} 
\end{thm}

Moreover, in the subcritical case, a central limit theorem is proved
in \cite{JLTV12} (see Theorem 3.8). 
In this work, we study the large deviations. 

\begin{defn}
For 
$\beta< \varphi_\alpha$ (resp.\ $\beta> \varphi_\alpha$), let 
$P(S_0,\beta)$ denote the tail probability that 
the initial susceptibility of $S_0\subset[n]$ in $\Gnp$
results in 
some number 
$|I_*|\le\beta\vartheta$ (resp.\ $|I_*|\ge \beta\vartheta$)
of eventually infected vertices.
\end{defn}

\begin{thm}\label{T_LD}
Let $p$ be as in \eqref{E_p}, $\alpha\in [0,1)$  
and $\beta\neq\varphi_\alpha\in[\alpha_r,1]$. 
Suppose that a set $S_0$ (independent of $\Gnp$) of size $|S_0|\sim\alpha_r\vartheta$
is initially susceptible.  
Then 
\[
\lim_{n\to\infty}\frac{1}{\vartheta}\log P(S_0,\beta)
=\xi(\alpha,\beta)\]
where
\[
\xi(\alpha,\beta)= 
-\beta^r/r+
\begin{cases}
(\beta-\alpha_r)[1+\log(\beta^r/r (\beta-\alpha_r)]&\beta\le \alpha\\
\alpha/r
-(r-2)(\beta-\alpha)
+(r-1)\log(\beta^\beta/\alpha^{\alpha_r})&\beta>\alpha.
\end{cases}
\]
\end{thm}

For any given $\alpha\in[0,1)$,   $\xi(\alpha,\beta)$ is increasing 
in $\beta\in[\alpha_r,\varphi_\alpha)$, decreasing in $\beta\in(\varphi_\alpha,1]$
(see \cref{A_xi}), and 
$\xi(\alpha,\varphi_\alpha)=0$ by \eqref{E_varphi}, in line with \cref{T_JLTV}.
The point $\vartheta$ (associated with $\beta=1$) is critical, see 
\cref{S_heuristics} below.  
As such, 
simply $\xi(\alpha,\beta)=\xi(\alpha,1)$ for $\beta>1$. 
See \cref{F_rate}. 

It might be interesting to investigate the nature of  
$\Gnp$, conditioned on the event that a given $S_0$ eventually infects
a certain number of vertices, or on the existence of such a set $S_0$. 

Finally, let us mention here that Torrisi et al.\ \cite{TGL18} analyze 
large deviations in the 
supercritical case, $\alpha>1$, where typically $|I_*|\sim n$.

\begin{figure}[h!]
\centering
\includegraphics[scale=0.3]{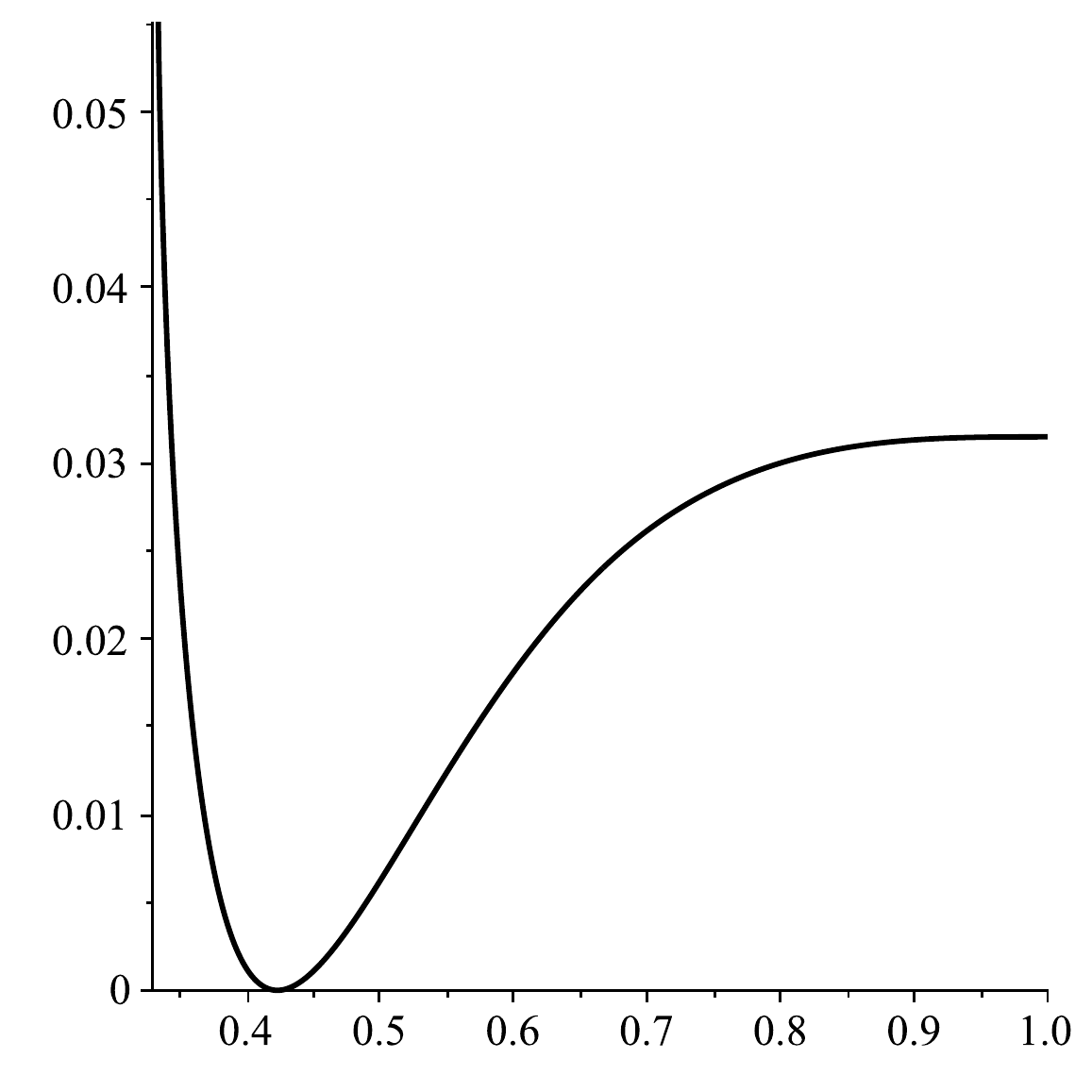}
\vspace{-0.25cm}
\caption{
For $r=2$ and $\alpha=2/3$, the rate function 
$-\xi(\alpha,\beta)$ is plotted as a function 
of $\beta$. 
}
\label{F_rate}
\end{figure}

\subsection{Heuristics}\label{S_heuristics}
We first briefly recall the heuristic for \cref{T_JLTV}
given in Section 6 of \cite{JLTV12}. 
Until $t_*$,
the number of susceptible vertices 
in step $t$ satisfies 
\begin{equation}\label{E_St}
|S_t|\sim{\rm Bin}(n-|S_0|,\pi_t)-t+|S_0|
\end{equation}
where $\pi_t=\P({\rm Bin}(t,p)\ge r)$. 
By the law of large numbers, 
with high probability $|S_t|\approx\E |S_t|$.
A calculation shows that if $|S_0| > (1-1/r)\vartheta$ 
then $\E |S_t| > t$ for $t < n-o(n)$.  On the other hand,  
if $|S_0|=\alpha_r\vartheta$, for some $\alpha<1$,  
then we have $\E |S_{\varphi_\alpha\vartheta}| \approx 0$.
To see this, note that $pt=O[(\vartheta/n)^{1/r}]\ll1$ for $t\le O(\vartheta)$
since $\vartheta\ll n$. Hence (see e.g.\ Section 8 of \cite{JLTV12}) we have
\[
\pi_t=\frac{(pt)^r}{r!}[1+O(pt+1/t)]
\]
and so 
\begin{equation}\label{E_typ}
\E |S_{x\vartheta}|/\vartheta 
\sim 
x^r/r-x+\alpha_r.
\end{equation}

Next, we describe a natural heuristic, using the 
Euler--Lagrange equation, that allows us to anticipate 
the least-cost, deviating trajectories, 
which lead to \cref{T_LD}. 
The proof, given in \cref{S_proof} below, makes this 
rigorous by a discrete analogue 
of the Euler--Lagrange equation established by Guseinov \cite{G04}. 
We think this method will be of use in 
studying the tail behavior of other random processes.

Consider a trajectory $y(x)\ge0$ 
from $\alpha_r$ to $0$ over $[0,\beta]$. 
Suppose that $|S_{x\vartheta}|/\vartheta$ has followed this trajectory
until step $t-1=x\vartheta$. 
In the next step $t$, some vertex $v_t\in S_{t-1}$ is infected. 
There are approximately 
a Poisson with mean $np^r{t-1\choose r-1}\approx x^{r-1}$ number 
of vertices that are neighbors with $v_t$ and 
$r-1$ of the $t-1$ vertices infected
in previous steps $s<t$. Such vertices become susceptible in step $t$. 
Therefore, to continue along this trajectory, we 
require this random variable to take the value 
\[
1+\vartheta[y(x+1/\vartheta)-y(x)]\approx1+y'(x).
\]
(The ``$+1$'' accounts for the vertex $v_t$ 
that is infected in step $t$, and so removed from the susceptible
set.)
This event has log probability approximately 
$-\Gamma^*_{x^{r-1}}(1+y'(x))$, where
\[
\Gamma^*_{\lambda}(u)=-u[1-\lambda/u+\log(\lambda/u)]
\]
is the Legendre--Fenchel transformation of the cumulant-generating function
of a mean $\lambda$ Poisson. 
Hence $|S_{x\vartheta}|/\vartheta\approx y(x)$ on $[0,\beta]$ with 
approximate log probability
\begin{equation}\label{E_int}
\vartheta\int_0^\beta (1+y'(x))[1-\frac{x^{r-1}}{1+y'(x)}+\log \frac{x^{r-1}}{1+y'(x)}]dx
\end{equation}
(cf.\ \eqref{E_logPub2} below). 
Maximizing  this integral is particularly simple, since the integrand depends
on $y'$, but not $y$. 
The Euler--Lagrange equation implies that the least-cost trajectory 
satisfies 
\[
\frac{d}{dx}\log \frac{x^{r-1}}{1+y'(x)}=0
\implies y(x)=(\beta-\alpha_r)(x/\beta)^r-x+\alpha_r
\]
except where possibly the boundary constraint $y(x)\ge0$ might intervene. 
Note that $y(\beta)=0$ for all $\alpha$. 
If $\beta>\alpha$, the above trajectory hits 0 before $\beta$, 
and is negative in between, and so is not viable. 
As it turns out, the least-cost trajectory is 
\begin{equation}\label{E_haty}
\hat y_{\alpha,\beta}(x)=
\begin{cases}
(\beta-\alpha_r)(x/\beta)^r-x+\alpha_r&\beta\le \alpha\\
[x^r/(r\alpha^{r-1})-x+\alpha_r]\1_{x\le \alpha}&\beta>\alpha.
\end{cases}
\end{equation}
By \eqref{E_varphi}, for $\beta>\alpha$, 
the trajectories $\hat y_{\alpha,\beta}$ and $\hat y_{\alpha,\alpha}$
agree on $[0,\alpha]$, and $\hat y_{\alpha,\beta}=0$ thereafter.
Setting $\beta=\varphi_\alpha$, we recover the typical, zero-cost  
trajectory \eqref{E_typ}.
See \cref{F_trajs}.  
Finally, note that substituting \eqref{E_haty} into \eqref{E_int}, 
we obtain $\vartheta\xi(\alpha,\beta)$.

\begin{figure}[h!]
\centering
\includegraphics[scale=0.3]{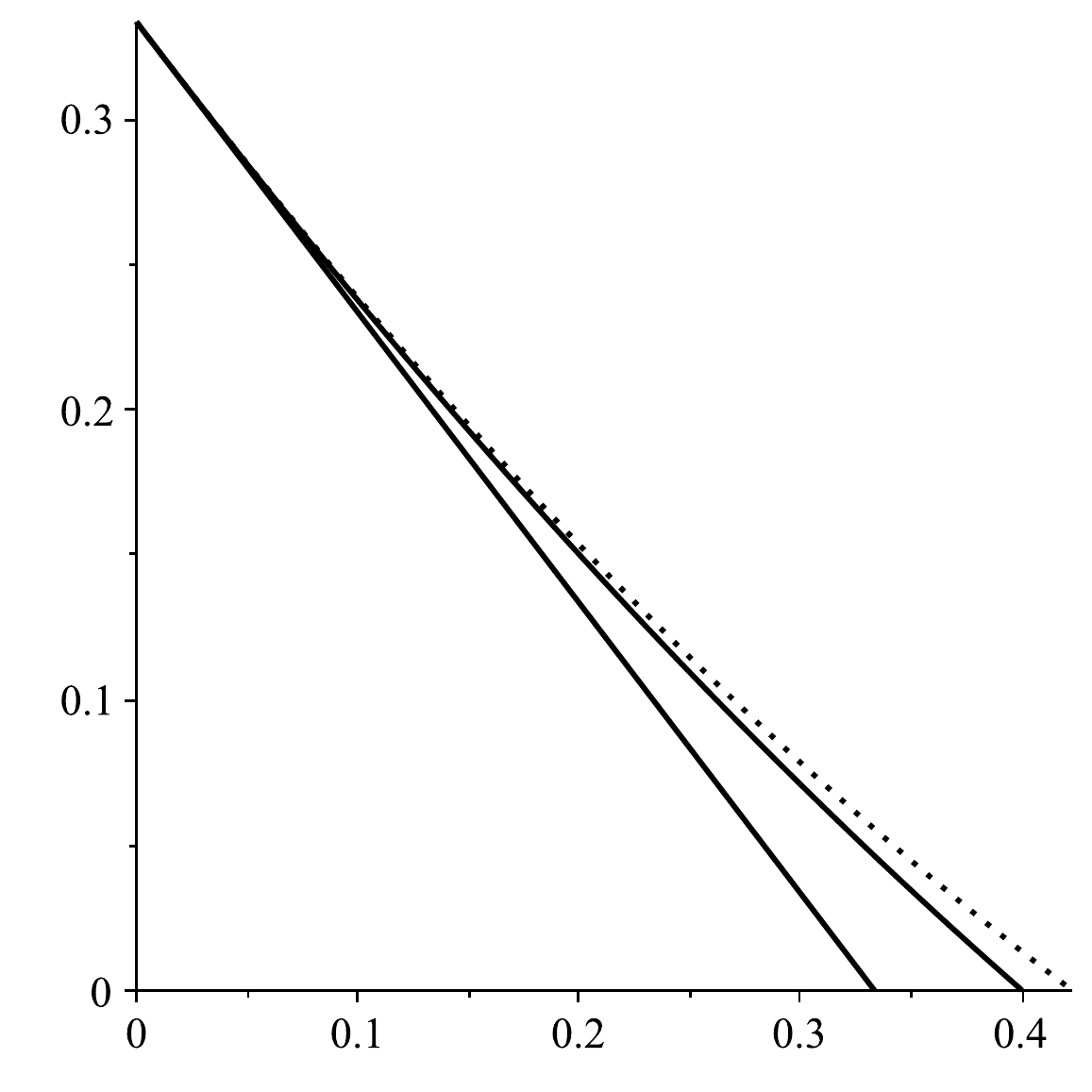}
\includegraphics[scale=0.3]{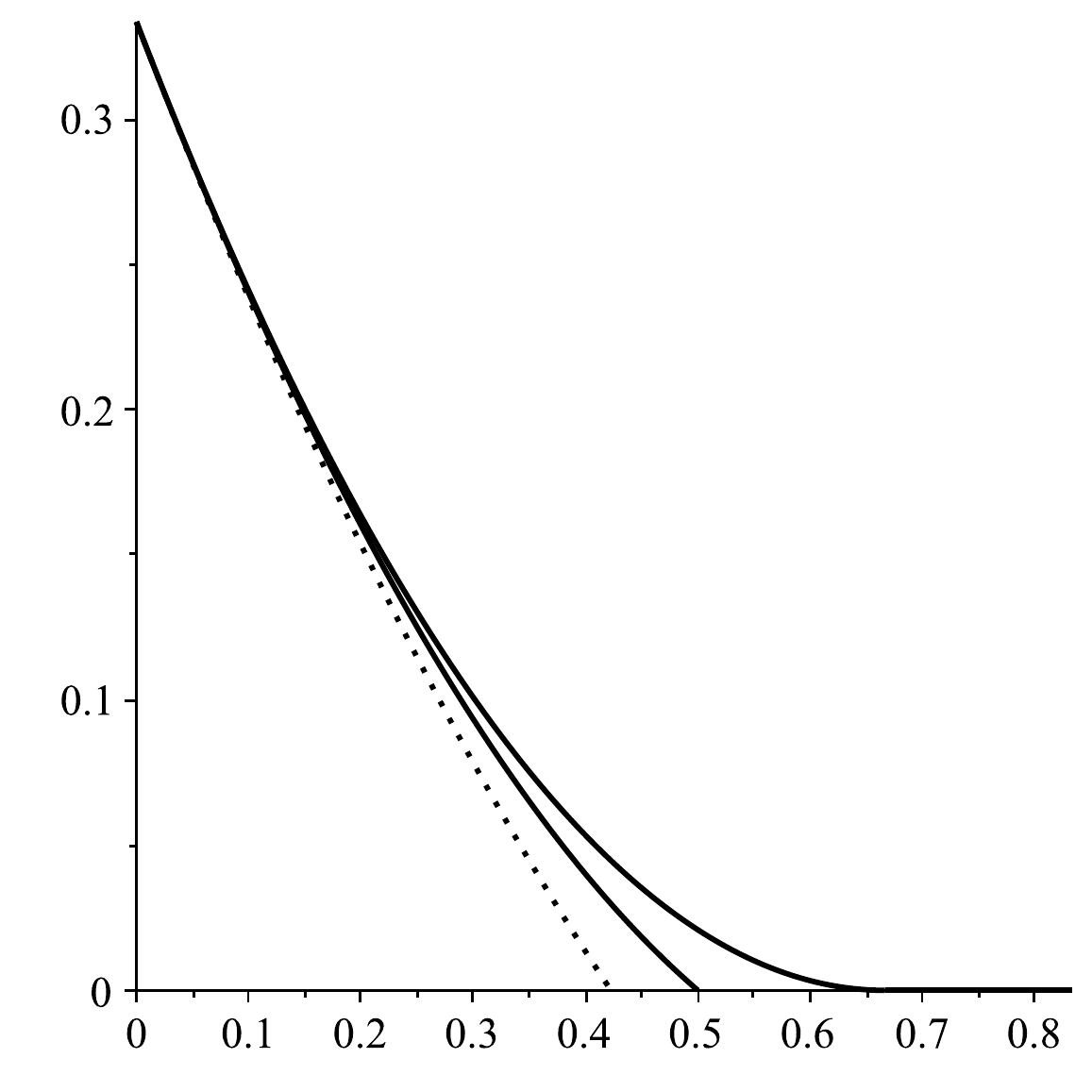}
\vspace{-0.25cm}
\caption{
In both figures, $r=2$ and $\alpha=2/3$. 
The typical, zero-cost trajectory 
appears as a dotted line.
Least-cost, deviating trajectories $\hat y_{\alpha,\beta}$ for $\beta=1/3,2/5$
appear at left and $\beta=1/2,2/3,5/6$ at right. 
}
\label{F_trajs}
\end{figure}

\subsection{Motivation}
We came to this problem in studying {\it $H$-bootstrap percolation,}
as introduced by Balogh et al.\ \cite{BBM12}, where any edge 
in an otherwise infected copy of $H$ in a graph 
becomes infected.
\cref{T_LD}, when $r=2$
and $\vartheta=\Theta(\log n)$, 
plays a role together with \cite{AK18,K17} in finding the precise leading order
asymptotics for the 
 critical probability $p_c$ 
on $\Gnp$, in the case that $H=K_4$, solving a problem in \cite{BBM12}. 

\subsection{An application}
A susceptible set $S_0$
is called {\it contagious} if it infects all of $\Gnp$ eventually  
(i.e., $I_*=[n]$). 
Such sets have been studied for various graphs, 
e.g.\ \cite{COFKR15,FPR18,GS15,M09}, 
and recently for $\Gnp$
by Feige et al.\ \cite{FKR16}. 

By \cref{T_JLTV}, $\Gnp$ has contagious sets
of size $\Theta(\vartheta)$, however, 
there exist contagious sets that are much smaller. 
Upper and lower bounds
for $m(\Gnp,r)$, the minimal size 
of a contagious set in $\Gnp$, 
are obtained in \cite{FKR16}. 
In \cite{AK18}, we showed that 
\begin{equation}\label{E_pc}
p_c\sim
[(r-1)!/n]^{1/r}[(\log{n})/(1-1/r)^2]^{1/r-1}
\end{equation}
 is the sharp
threshold for contagious sets of the smallest
possible size $r$. 

\cref{T_LD}
yields lower bounds for $m(\Gnp,r)$ for $p<p_c$
that sharpen those in \cite{FKR16} by a 
linear, multiplicative factor in $r$. 
Of course, {\it finding} sets of this size (if they exist)
is a difficult and interesting question
(cf.\ the NP-complete problem of target set selection 
from viral marketing \cite{DR01,KKT15}). 

\begin{corollary}\label{T_mGr}
Suppose that, for some  $1\ll \vartheta\ll n$, 
\[
p=[(r-1)!/n]^{1/r}[\vartheta/(1-1/r)^2]^{1/r-1}.
\]
Then, with high probability,   
\[
m(\Gnp,r)
\ge
(1-o(1))
r\vartheta/\log(n/\vartheta).
\] 
\end{corollary}

Since $\vartheta\ll n$, the right hand side  
is $\ll \vartheta$. 
This result follows by an easy union bound, applying \cref{T_LD}
in the case that $\alpha=0$ and $\beta=1$, 
see \cref{A_mGr}.

By \cite{AK18} this lower bound is sharp for $p$ close
to $p_c$ in \eqref{E_pc}. It seems possible that the methods in \cite{AK18} 
could show it is also sharp at least for $\vartheta\le O(\log{n})$.

\section{Proof of \cref{T_LD}}
\label{S_proof}

Suppose that $|S_0|/\vartheta=\alpha_n\to\alpha_r$. 
For $\beta< \varphi_\alpha$ (resp.\ $\beta> \varphi_\alpha$), 
let ${\mathcal E}_n$ be the event that 
$|I_*|/\vartheta\le\beta$
(resp.\ $|I_*|/\vartheta\ge \beta$). 
Note that ${\mathcal E}_n$ occurs if and only if 
$|S_{x\vartheta}|=0$ at some $x\le \beta$ (resp. $|S_{x\vartheta}|>0$
for all $x< \beta$). 
We first upper bound $\P({\mathcal E}_n)$. 
The lower bound is much easier, and discussed 
at the end of the proof. 
Our strategy is to use a discretized 
Euler--Lagrange equation to identify the
optimal trajectory amongst those realizing ${\mathcal E}_n$.

Let ${\mathcal Y}_{n}$ denote the 
set of trajectories $(y_0,y_1,\ldots,y_m)$ of $|S_{x\vartheta}|/\vartheta$, 
starting from $y_0=\alpha_n$, along 
$m=\Theta[\vartheta/(\log\vartheta)^2]$
points $0=x_0<x_1<\cdots<x_m=\min\{1,t_*/\vartheta\}$
that are compatible with ${\mathcal E}_n$.
That is, we consider trajectories of $|S_t|$ 
realizing ${\mathcal E}_n$
along a discrete set of points
until time $t=\vartheta$
or the first time $t_*$ that $S_t=0$, if that comes first. 
Moreover, we assume that the $x_i$ are 
spaced as evenly as possible, 
subject to all $x_i\vartheta\in\Z$.  
By standard 
tail estimates, we may assume that in ${\mathcal Y}_{n}$ all 
$y_i\le O(1)$. 
Indeed, for $t\le O(\vartheta)$, 
by \eqref{E_St} and 
Chernoff's bound, 
\[
\log \P( |S_{t}| \ge (1+\delta)\vartheta )
\le -O(\delta^2 \vartheta).
\]
Therefore, for $A$ sufficiently large, the log
probability that any $|S_t|>A\vartheta$ while $t\le \vartheta$
is less than $\vartheta\xi(\alpha,\beta)$. 
Hence $|{\mathcal Y}_{n}|\le (A\vartheta)^m$, and so by the choice of $m$,
\begin{equation}\label{E_logPub}
\frac{1}{\vartheta}\log \P({\mathcal E}_{n})
\le o(1)+\frac{1}{\vartheta}\sum_{i=0}^{m-1} \log
\P(\frac{|S_{x_{i+1}\vartheta}|}{\vartheta}=\hat y_{i+1} | \frac{|S_{x_{i}\vartheta}|}{\vartheta}=\hat y_{i})
\end{equation}
where $\hat y$ maximizes the right hand side over $y\in{\mathcal Y}_{n}$. 

Next, we aim to identify $\hat y$. 
By \eqref{E_St}, note that 
\[
\Delta |S_{x_i\vartheta}|\sim {\rm Bin}(n-|S_0|,\Delta\pi(x_i\vartheta))-\vartheta\Delta x_i,
\]
where $\Delta$ is the forward difference operator, 
$\Delta g_i=g_{i+1}-g_i$. 
Hence
\begin{multline*}
\P(\frac{|S_{x_{i+1}\vartheta}|}{\vartheta}=y_{i+1} | \frac{|S_{x_{i}\vartheta}|}{\vartheta}=y_{i})\\
\le 
\left(e\frac{n\Delta\pi(x_i\vartheta)}{\vartheta(\Delta x_i+\Delta y_i)}\right)^{\vartheta(\Delta x_i+\Delta y_i)}
[1-\Delta\pi(x_i\vartheta)]^{n-|S_0|-\vartheta(\Delta x_i+\Delta y_i)}.
\end{multline*}

We require a technical result. The proof is elementary, though
somewhat tedious, see \cref{S_lem} below. 
Note that by \eqref{E_p}, $1\ll\vartheta\ll1/p$.

\begin{lem}\label{L_delpi}
We have that 
\[
\frac{rn}{\vartheta}\frac{\Delta\pi(x_i\vartheta)}{\Delta(x_i^r)}
=1+O(p\vartheta+\frac{1}{\log\vartheta})\sim 1.
\]
\end{lem}

Using this (and the standard bound $1-x\le e^{-x}$), it follows that 
\begin{multline*}
\frac{1}{\vartheta}\sum_{i=0}^{m-1} \log
\P(\frac{|S_{x_{i+1}\vartheta}|}{\vartheta}= y_{i+1} | \frac{|S_{x_{i}\vartheta}|}{\vartheta}= y_{i})\\
\le o(1)+\sum_{i=0}^m(\Delta x_i+\Delta y_i)[
1-\frac{\Delta(x_i^r)/r}{\Delta x_i+\Delta y_i}+\log\frac{\Delta(x_i^r)/r}{\Delta x_i+\Delta y_i}].
\end{multline*}
Note that $\log x-x$ is increasing for $x\in(0,1]$, and $\Delta(x_i^r)/r\le x_{i+1}^{r-1}\Delta x_i$. 
Therefore, by \eqref{E_logPub}, we find that 
\begin{equation}\label{E_logPub2}
\frac{1}{\vartheta}\log \P({\mathcal E}_{n})
\le o(1)+\sum_{i=0}^{m-1}f(x_{i+1},\Delta \hat y_i/\Delta x_i)\Delta x_i 
\end{equation}
where
\[
f(u,v)=(1+v)[1-\frac{u^{r-1}}{1+v}+\log\frac{u^{r-1}}{1+v}]
\]
(cf.\ \eqref{E_int} above). 
In obtaining an upper bound for $\P({\mathcal E}_n)$, 
we may now lift the restriction that all $y_i\vartheta\in\Z$. 
Then, by the discrete Euler--Lagrange equation in \cite{G04} (see Theorem 5), 
the maximizer $\hat y$ satisfies
\[
\Delta f_v(x_{i+1},\Delta \hat y_i/\Delta x_i)
\equiv 0.
\]
Since 
\[
f_v(u,v)=\log\frac{u^{r-1}}{1+v}
\]
this implies that
$1+\Delta \hat y_i/\Delta x_i=bx_{i+1}^{r-1}$, 
for some constant $b$, between any two points $x_j<x_k$
where $\hat y_i>0$ for $j<i<k$. On the other hand, if both 
$\hat y_j=\hat y_k=0$, then necessarily $\hat y_i=0$
for $j<i<k$. 
By standard results on the Euler approximation  
of differential equations (see e.g.\ Theorem 7.3 and 7.5 in Section I.7 of \cite{HNW93}), 
it follows that the discrete derivative $\Delta \hat y_i/\Delta x_i$
is within $O(1/m)$ of the function $bx^{r-1}-1$, for some $b$,  
on all segments where $\hat y_i>0$. Hence, by \eqref{E_logPub2}
and the continuity of $f$,  
\begin{equation}\label{E_logPub3}
\lim_{n\to\infty}\frac{1}{\vartheta}\log \P({\mathcal E}_{n})
\le \sup_{y\in {\mathcal Y}}\int_0^{\beta_y}f(x,y'(x))dx
\end{equation}
where ${\mathcal Y}$ is the set of non-negative, continuous functions $y$ 
on some $[0,\beta_y]\subset[0,1]$ 
of the form $bx^r/r-x+a$ on all segments where $y>0$, 
such that $y(0)=\alpha_r$ and 
$\beta_y\le\beta$ and $y(\beta_y)=0$ if $\beta<\varphi_\alpha$
(resp.\ $\beta_y\ge\beta$ if $\beta>\varphi_\alpha$). 

Next, to further reduce the possibilities for the maximizer $\hat y\in{\mathcal Y}$, 
we observe that 
$\hat y$ is necessarily non-increasing. This is intuitive, since the process is sub-critical while the total
number of infected vertices remains less than $\vartheta$.
To see this
formally, note that
$bx^{r-1}-1\le0$ 
for any $x\le 1$ unless $b>1$, and 
\[
f(x,bx^{r-1}-1)=(b-1-b\log b)x^{r-1}
\]
is decreasing in $b>1$. 
It follows that, in either case $\beta<\varphi_\alpha$ or 
$\beta>\varphi_\alpha$, we may assume
$\hat y(\beta_y)=0$. 
Therefore, 
to complete the proof, we need only show that 
\[
\int_0^\beta f(x,y'_{\alpha,\beta}(x))dx=\xi(\alpha,\beta)
\]
(with $y_{\alpha,\beta}$ as in \eqref{E_haty} above)
is increasing in $\beta<\varphi_\alpha$ and decreasing 
in $\beta>\varphi_\alpha$. 
This follows by simple calculus, 
see \cref{A_xi} below. 

The lower bound is much simpler, 
so we omit the details. 
It suffices to consider any possible trajectory for $|S_t|$ sufficiently close to the optimal 
$\hat y\vartheta$. 
For instance, 
consider the case that all 
$|S_{x_i\vartheta}|=\lceil \hat y(x_i)\vartheta\rceil$ if $\beta<\varphi_\alpha$
(resp.\ 
$|S_{x_i\vartheta}|=\lceil(\hat y(x_i)+\Delta x_i)\vartheta\rceil$
if $\beta>\varphi_\alpha$). Note that ``$+\Delta x_i$'' is added in the latter case to ensure
that $|S_t|>0$ between increments.

\appendix

\section{Technical results}

\subsection{Shape of $\xi$}\label{A_xi}

In this section, we prove the claim that $\xi(\alpha,\beta)$ is increasing 
in $\beta\in[\alpha_r,\varphi_\alpha)$ and decreasing in $\beta\in(\varphi_\alpha,1]$. 
First note that, for $\beta\le\alpha$, 
\[
\frac{\partial }{\partial\beta}\xi(\alpha,\beta)
=
\log\frac{\beta^r/r}{\beta-\alpha_r}+r(1-\alpha_r/\beta)-\beta^{r-1}.
\]
Therefore, if $\beta\in[\alpha_r,\varphi_\alpha]$, by \eqref{E_varphi} and $\log x\ge 1-1/x$,
\[
\frac{\partial }{\partial\beta}\xi(\alpha,\beta)
\ge
\frac{1-\beta^{r-1}}{\beta^r/r}(\alpha_r-\beta+\beta^r/r)\ge0.
\] 
On the other hand, if $\beta\in[\varphi_\alpha,\alpha]$, 
by \eqref{E_varphi} and $\log x\le x-1$, 
\[
\frac{\partial }{\partial\beta}\xi(\alpha,\beta)
\le
\frac{(r-1)(\alpha-\beta)}{\beta(\beta-\alpha_r)}(\alpha_r-\beta+\beta^r/r)\le 0. 
\]
Finally, if $\beta\in[\alpha,1]$, note that 
\[
\frac{\partial }{\partial\beta}\xi(\alpha,\beta)
=1-\beta^{r-1}+(r-1)\log\beta
\le 1-\beta+\log\beta\le0. 
\]

\subsection{Lower bound for $m(\Gnp,r)$}\label{A_mGr}

\begin{proof}[Proof of \cref{T_mGr}]
For $\delta>0$, let $t_\delta=(1-\delta)r\vartheta/\log(n/\vartheta)$.
We show that, with high probability, $\Gnp$ 
has no contagious sets  
smaller than $t_\delta$.
Note that 
\[
\xi(0,1)\vartheta/(1-1/r)^2=-r\vartheta.
\]
The expected number of subsets $S_0\subset[n]$
of size $|S_0|=t_\delta$ which if initially susceptible
cause $|I_*|\ge \vartheta/(1-1/r)^2$ vertices to 
be infected eventually is at most 
\[
{n\choose t_\delta} e^{-r\vartheta (1+o(1))}
\le
(ne/t_\delta)^{t_\delta}
e^{-r\vartheta (1+o(1))}
=e^{-r\vartheta \psi}, 
\] 
where
\[
\psi=
1+o(1)
-(1-\delta)\log(ne/t_\delta)/\log(n/\vartheta).
\]
Since
\[
\log(ne/t_\delta)\le \log(n/\vartheta)
+O\left(\log\log(n/\vartheta)\right), 
\]
$\psi>0$ for all large $n$,
and the result follows. 
\end{proof}

\subsection{Increments of $\pi$}\label{S_lem}

\begin{proof}[Proof of \cref{L_delpi}]
Recall that 
\[
m=\Theta(\frac{1}{\Delta x_i})=\Theta(\frac{\vartheta}{(\log\vartheta)^2}).
\]

When $i=0$, we have $\Delta \pi(x_i\vartheta)=\pi(x_1\vartheta)$
since $x_0=0$. By the estimates 
discussed in \cref{S_heuristics}, 
\[
\frac{rn}{\vartheta}\pi(x_1\vartheta)
= x_1^r[1+O(p\vartheta+\frac{1}{(\log\vartheta)^2})].
\]

Next, we assume that $i\ge1$.
Then 
$x_{i+1}\le O(x_i)$ and, for all $\ell\ge r$, 
\begin{equation}\label{E_delxiell}
1\le \frac{\Delta(x_i^\ell)}{\ell x_i^{\ell-1}\Delta x_i}\le O(1)^\ell.
\end{equation}

For the lower bound, first note that 
\[
\P({\rm Bin}(x_{i+1}\vartheta,p)>r)> \P({\rm Bin}(x_{i}\vartheta,p)>r)
\]
and so 
\[
\Delta\pi(x_i\vartheta)
> \Delta\P({\rm Bin}(x_{i}\vartheta,p)=r).
\]

Hence, using \eqref{E_p} and 
\eqref{E_delxiell} (and the standard bounds $(n-k)^k\le {n\choose k}k!\le n^k$
and $(1-x)^y\ge 1-xy$) we find 
\begin{align*}
\frac{rn}{\vartheta}\Delta\pi(x_i\vartheta)
&\ge 
(1-p)^{x_{i}\vartheta-r}[
(x_{i+1}-\frac{r}{\vartheta})^r
(1-p)^{\vartheta\Delta x_{i}}-
x_i^r]\\
&\ge 
\Delta (x_{i}^r)
(1-p\vartheta)[1-
\frac{x_{i+1}^r}{\Delta (x_{i}^r)}(p\vartheta\Delta x_{i}+\frac{r^2}{x_{i+1}\vartheta})]\\
&=
\Delta (x_{i}^r)
[1-
O(p\vartheta+\frac{1}{(\log\vartheta)^2})].
\end{align*}

The upper bound requires slightly more attention.
Note that, by the choice of $m$, 
$\log m\ll x_1\vartheta$. Therefore $\log m\le x_i\vartheta$
for all large $n$. 
Hence, for all large $n$, 
\[
\Delta\pi(x_i\vartheta)
<\P({\rm Bin}(x_{i+1}\vartheta,p)>\log m)
+\sum_{\ell=r}^{\log m}\Delta\P({\rm Bin}(x_{i}\vartheta,p)=\ell).
\]
Since $p\vartheta\ll1\ll m$, for all large $n$, 
\[
\P({\rm Bin}(x_{i+1}\vartheta,p)> \log m)
\le \vartheta(x_{i+1}p\vartheta)^{1+\log m}.
\]
Therefore, by \eqref{E_p}, \eqref{E_delxiell} and the choice of $m$, 
\[
\frac{rn}{\vartheta \Delta (x_{i}^r)}
\P({\rm Bin}(x_{i+1}\vartheta,p)> \log m)
\le O( \frac{n}{\Delta x_i}(p\vartheta)^{1+\log m})
\ll p\vartheta.
\]
Next, by \eqref{E_delxiell}, it follows
that, for all $\ell\le \log m$ and large $n$, 
\begin{align*}
\Delta\P({\rm Bin}(x_{i}\vartheta,p)=\ell)
&\le\frac{(p\vartheta)^\ell}{\ell!}[x_{i+1}^\ell-x_i^\ell(1-\frac{\ell}{x_i\vartheta})^\ell]
\\
&\le\frac{(p\vartheta)^\ell}{\ell!}\Delta(x_i^\ell)(1
+\frac{\ell}{\vartheta\Delta x_i}).
\end{align*}
Therefore, by \eqref{E_p} and \eqref{E_delxiell}, for all large $n$, 
\begin{align*}
\frac{rn}{\vartheta}\sum_{\ell=r}^{\log m}\Delta\P({\rm Bin}(x_{i}\vartheta,p)=\ell)
&\le
\sum_{\ell=r}^{\log m}\frac{r!(p\vartheta)^{\ell-r}}{\ell!}\Delta(x_i^\ell)
(1+\frac{\ell}{\vartheta\Delta x_i})
\\
&\le
\Delta(x_i^r)
(1+\frac{\log m}{\vartheta\Delta x_i})[1
+\sum_{\ell>0}O(p\vartheta)^\ell]\\
&\le \Delta(x_i^r)
[1+O(p\vartheta+\frac{1}{\log\vartheta})].
\end{align*}

Altogether, we find that
\[
\frac{rn}{\vartheta}\Delta\pi(x_i\vartheta)
=
\Delta(x_i^r)[1+O(p\vartheta+\frac{1}{\log\vartheta})]
\]
as claimed. 
\end{proof}

\providecommand{\bysame}{\leavevmode\hbox to3em{\hrulefill}\thinspace}
\providecommand{\MR}{\relax\ifhmode\unskip\space\fi MR }
\providecommand{\MRhref}[2]{%
  \href{http://www.ams.org/mathscinet-getitem?mr=#1}{#2}
}
\providecommand{\href}[2]{#2}


\begin{thebibliography}{10}

\bibitem{J91}
J.~Adler, \emph{Bootstrap percolation}, Physica A \textbf{171} (1991),
  453--470.

\bibitem{JL03}
J.~Adler and U.~Lev, \emph{Bootstrap percolation: Visualizations and
  applications}, Braz. J. Phys. \textbf{33} (2003), 641--644.

\bibitem{AL88}
M.~Aizenman and J.~L. Lebowitz, \emph{Metastability effects in bootstrap
  percolation}, J. Phys. A \textbf{21} (1988), no.~19, 3801--3813.

\bibitem{AK18}
O.~Angel and B.~Kolesnik, \emph{Sharp thresholds for contagious sets in random
  graphs}, Ann. Appl. Probab. \textbf{28} (2018), no.~2, 1052--1098.

\bibitem{BBDCM12}
J.~Balogh, B.~Bollob{\'a}s, H.~Duminil-Copin, and R.~Morris, \emph{The sharp
  threshold for bootstrap percolation in all dimensions}, Trans. Amer. Math.
  Soc. \textbf{364} (2012), no.~5, 2667--2701.

\bibitem{BBM12}
J.~Balogh, B.~Bollob{\'a}s, and R.~Morris, \emph{Graph bootstrap percolation},
  Random Structures Algorithms \textbf{41} (2012), no.~4, 413--440.

\bibitem{CLR79}
J.~Chalupa, P.~L. Leath, and G.~R. Reich, \emph{Bootstrap percolation on a
  bethe lattice}, J. Phys. C \textbf{21} (1979), L31--L35.

\bibitem{COFKR15}
A.~Coja-Oghlan, U.~Feige, M.~Krivelevich, and D.~Reichman, \emph{Contagious
  sets in expanders}, Proceedings of the {T}wenty-{S}ixth {A}nnual {ACM}-{SIAM}
  {S}ymposium on {D}iscrete {A}lgorithms, SIAM, Philadelphia, PA, 2015,
  pp.~1953--1987.

\bibitem{DR01}
P.~Domingos and M.~Richardson, \emph{Mining the network value of customers},
  Proceedings of the Seventh ACM SIGKDD International Conference on Knowledge
  Discovery and Data Mining (New York, NY, USA), KDD '01, Association for
  Computing Machinery, 2001, p.~57–66.

\bibitem{ER59}
P.~Erd{\H{o}}s and A.~R{\'e}nyi, \emph{On random graphs. {I}}, Publ. Math.
  Debrecen \textbf{6} (1959), 290--297.

\bibitem{FKR16}
U.~Feige, M.~Krivelevich, and D.~Reichman, \emph{Contagious sets in random
  graphs}, Ann. Appl. Probab. \textbf{27} (2017), no.~5, 2675--2697.

\bibitem{FPR18}
D.~Freund, M.~Poloczek, and D.~Reichman, \emph{Contagious sets in dense
  graphs}, European J. Combin. \textbf{68} (2018), 66--78.

\bibitem{GS15}
A.~Guggiola and G.~Semerjian, \emph{Minimal contagious sets in random regular
  graphs}, J. Stat. Phys. \textbf{158} (2015), no.~2, 300--358.

\bibitem{G04}
G.-S. Guseinov, \emph{Discrete calculus of variations}, Global analysis and
  applied mathematics, AIP Conf. Proc., vol. 729, Amer. Inst. Phys., Melville,
  NY, 2004, pp.~170--176.

\bibitem{HNW93}
E.~Hairer, S.~P. N{\o}rsett, and G.~Wanner, \emph{Solving ordinary differential
  equations. {I}}, second ed., Springer Series in Computational Mathematics,
  vol.~8, Springer-Verlag, Berlin, 1993, Nonstiff problems.

\bibitem{H03}
A.~E. Holroyd, \emph{Sharp metastability threshold for two-dimensional
  bootstrap percolation}, Probab. Theory Related Fields \textbf{125} (2003),
  no.~2, 195--224.

\bibitem{JLTV12}
S.~Janson, T.~{\L}uczak, T.~Turova, and T.~Vallier, \emph{Bootstrap percolation
  on the random graph {$G\sb {n,p}$}}, Ann. Appl. Probab. \textbf{22} (2012),
  no.~5, 1989--2047.

\bibitem{KKT15}
J.~Kempe, D.and~Kleinberg and \'{E}. Tardos, \emph{Maximizing the spread of
  influence through a social network}, Theory Comput. \textbf{11} (2015),
  105--147.

\bibitem{K17}
B.~Kolesnik, \emph{Sharp threshold for ${K}_4$-percolation}, preprint (2017),
  available at \href{https://arxiv.org/abs/1705.08882}{arXiv:1705.08882}.

\bibitem{M09}
R.~Morris, \emph{Minimal percolating sets in bootstrap percolation}, Electron.
  J. Combin. \textbf{16} (2009), no.~1, Research Paper 2, 20.

\bibitem{PRK75}
M.~Pollak and I.~Riess, \emph{Application of percolation theory to 2d-3d
  {H}eisenberg ferromagnets}, Physica Status Solidi (b) \textbf{69} (1975),
  no.~1, K15--K18.

\bibitem{ST85}
G.-P. Scalia-Tomba, \emph{Asymptotic final-size distribution for some
  chain-binomial processes}, Adv. in Appl. Probab. \textbf{17} (1985), no.~3,
  477--495.

\bibitem{S92}
R.~H. Schonmann, \emph{On the behavior of some cellular automata related to
  bootstrap percolation}, Ann. Probab. \textbf{20} (1992), no.~1, 174--193.

\bibitem{S83}
T.~Sellke, \emph{On the asymptotic distribution of the size of a stochastic
  epidemic}, J. Appl. Probab. \textbf{20} (1983), no.~2, 390--394.

\bibitem{TGL18}
G.~L. Torrisi, M.~Garetto, and E.~Leonardi, \emph{A large deviation approach to
  super-critical bootstrap percolation on the random graph {$G_{n,p}$}},
  Stochastic Process. Appl. \textbf{129} (2019), no.~6, 1873--1902.

\bibitem{V07}
T.~Vallier, \emph{Random graph models and their applications}, Ph.D. thesis,
  Lund Univ., 2007.

\end{thebibliography}
\end{document}